\documentclass[a4 paper, 12pt]{article}

\usepackage[usenames]{color}
\usepackage{bbm,amsmath,amssymb,amsfonts,amsthm,amsbsy,graphics,graphicx,xcolor}

\newtheorem{lem}{Lemma}
\newtheorem{thm}[lem]{Theorem}
\newtheorem{cor}[lem]{Corollary}

\theoremstyle{definition}
\newtheorem*{remark}{Remark}

\DeclareMathOperator{\C}{\mathbb{C}}
\DeclareMathOperator{\spa}{\mathrm{span}}
\DeclareMathOperator{\dimn}{\mathrm{dim}}
\DeclareMathOperator{\GL}{\mathrm{GL}}
\DeclareMathOperator{\Hom}{\mathrm{Hom}}
\DeclareMathOperator{\sig}{\sigma}
\DeclareMathOperator{\vp}{\varphi}
\DeclareMathOperator{\ii}{\mathrm{i}}

\newcommand{\dugi}[2]{\big[#1:#1[#2]\big]}
\newcommand{\drm}[1]{\mathrm{#1}}
\newcommand{\dbf}[1]{\mathbf{#1}}

\newcommand{\dmf}[1]{\mathfrak{#1}}
\newcommand{\dmb}[1]{\boldsymbol{#1}}
\newcommand{\dcal}[1]{\mathcal{#1}}

\title{\bf\upshape Graph Automorphisms from the Geometric Viewpoint}
\date{}

\author{ Wen-Xue Du ~~ Yi-Zheng Fan \\
{\small  \it School of Mathematical Science, Anhui University, Hefei,
230039, China}\\
{\footnotesize E-mail:  wenxuedu@gmail.com (W.-X. Du), fanyz@ahu.edu.cn (Y.-Z.
Fan)}}

\begin{document}
\maketitle

\begin{abstract} An automorphism of a graph $G=(V,E)$ is a bijective map $\phi$ from $V$ to itself such that $\phi(v_i)\phi(v_j)\in E$ $\Leftrightarrow$ $v_i v_j\in E$ for any two vertices $v_i$ and $v_j$. Denote by $\mathfrak{G}$ the group consisting of all automorphisms of $G$. Apparently, an automorphism of $G$ can be regarded as a permutation on $[n]=\{1,\ldots,n\}$, provided that $G$ has $n$ vertices. For each permutation $\sigma$ on $[n]$, there is a natural action on any given vector $\boldsymbol{u}=(u_1,\ldots,u_n)^t\in \mathbb{C}^n$ such that $\sigma\boldsymbol{u}=(u_{\sigma^{-1}1},u_{\sigma^{-1}2},\ldots,u_{\sigma^{-1} n})^t$, so $\sigma$ can be viewed as a linear operator on $\mathbb{C}^n$. Accordingly, one can formulate a characterization to the automorphisms of $G$, {\it i.e.,} $\sigma$ is an automorphism of $G$ if and only if every eigenspace of $\mathbf{A}(G)$ is $\sigma$-invariant, where $\mathbf{A}(G)$ is the adjacency matrix of $G$. Consequently, every eigenspace of $\mathbf{A}(G)$ is $\mathfrak{G}$-invariant, which is equivalent to that for any eigenvector $\boldsymbol{v}$ of $\mathbf{A}(G)$ corresponding to the eigenvalue $\lambda$, $\mathrm{span}(\mathfrak{G}\boldsymbol{v})$ is a subspace of the eigenspace $V_{\lambda}$. By virtue of the linear representation of the automorphism group $\mathfrak{G}$, we characterize those extremal vectors $\boldsymbol{v}$ in an eigenspace of $\mathbf{A}(G)$ so that $\mathrm{dim}~\mathrm{span}(\mathfrak{G}\boldsymbol{v})$ can attain extremal values, and furthermore, we determine the exact value of $\mathrm{dim}~\mathrm{span}(\mathfrak{G}\boldsymbol{v})$  for any eigenvector $\boldsymbol{v}$ of $\mathbf{A}(G)$. \end{abstract}



\section{Introduction}

Throughout this paper, $G$ stands for a simple graph with vertex set $V$ and edge set $E$. A bijective map $\phi:V\rightarrow V$ is called an {\em automorphism} of $G$ if $\phi(v_i)\phi(v_j)\in E$ $\Leftrightarrow$ $v_i v_j\in E$ for any two vertices $v_i$ and $v_j$. The group consisting of all automorphisms of $G$ is denoted by $\drm{Aut}~G$, which is called the automorphism group of $G$.\vspace{2mm}

The {\em adjacency matrix} of a graph $G$ on $n$ vertices is a $n\times n$ (0,1)-matrix where each entry $a_{ij}$ of the matrix is equal to 1 if and only if the two vertices $v_i$ and $v_j$ are adjacent. We denote the adjacency matrix of $G$ by $\dbf{A}(G)$, or $\dbf{A}$ for short in the case that one can easily identify the corresponding graph from the context. By virtue of the adjacency matrix of $G$, we can formulate an algebraic way of characterizing automorphisms of $G$.

A {\em permutation} of the set $[n]=\{1,2,\ldots,n\}$ is a bijective map $\sigma$ from $[n]$ to itself, and $\dmf{S}_n$ stands for the set consisting of all permutations of $[n]$, which is called the {\em symmetric group of degree $n$}.
Naturally, each permutation $\sigma\in\dmf{S}_n$ can act on a given vector $\dmb{u}=(u_1,\ldots,u_n)^t\in \C^n$, {\it i.e.,} $\sigma\dmb{u}=(u_{\sigma^{-1}1},u_{\sigma^{-1}2},\ldots,u_{\sigma^{-1} n})^t$, where $\C^n$ is the $n$-dimensional vector space over the complex field $\C$. Accordingly, any permutation $\sigma$ in $\dmf{S}_n$ can be regarded, through the action on vectors, as a linear operator on $\C^n$, which is denoted by $\dcal{T}_{\sigma}.$

We call a square matrix a {\em permutation matrix} if exactly one entry in each row and column is equal to 1, and all other entries are equal to 0. It is easy to check that the matrix $\dbf{P}_{\sigma}$ of the operator $\dcal{T}_{\sigma}$ with respect to the standard basis $\dmb{e}_1,\ldots,\dmb{e}_n$ is a permutation matrix, where each $\dmb{e}_i$ $(i=1,\ldots,n)$ has exactly one non-trivial entry on $i$th coordinate which is equal to 1, and all other entries of $\dmb{e}_i$ are equal to 0.

Let $\phi$ be a bijective map from the vertex set of a graph $G$ with $n$ vertices to itself. One can easily build in accordance with $\phi$ a permutation $\sigma_{\phi}\in\dmf{S}_n$ so that $\sigma_{\phi}(i)=j$ if and only if $\phi(v_i)=v_j$, where $v_i, v_j\in V(G)$, and vice versa. Consequently, it is legitimate not to distinguish two maps $\phi$ and $\sigma_{\phi}$, and furthermore, a moment's reflection will show that
$$\phi\mbox{ is an automorphism of }G\mbox{ if and only if }\dbf{P}_{\sigma_{\phi}}^{-1}\dbf{A}\dbf{P}_{\sigma_{\phi}}=\dbf{A},$$
which presents an algebraic way of characterizing automorphisms of $G$. Since every permutation $\sigma$ is an element of $\dmf{S}_n$, we can regard the group $\drm{Aut}~G$ as a subgroup of $\dmf{S}_n$. \vspace{0.5mm}

Evidently, the adjacency matrix $\dbf{A}(G)$ is symmetric and thus $\dbf{A}(G)$ can be viewed as the matrix of a self-adjoint operator $\dcal{T}_G$ on $\C^n$ with respect to an ordered basis $\dmb{b}_1,\ldots,\dmb{b}_n$, {\it i.e.,} $\dcal{T}_G(\dmb{b}_1,\ldots,\dmb{b}_n)=(\dmb{b}_1,\ldots,\dmb{b}_n)\dbf{A}$.  Accordingly,
$$\dcal{T}_G \dmb{b}_i=\sum_{v_j\sim v_i} \dmb{b}_j,~~~i=1,\ldots,n,$$
where the symbol $v_j\sim v_i$ indicates that the two vertices $v_j$ and $v_i$ are adjacent in $G$, and  $1\le i, j \le n$. So $\dcal{T}_G$ provides  the adjacency information about the graph $G$ and thus the standard basis $\dmb{e}_1,\ldots,\dmb{e}_n$ would be appropriate for obtaining the matrix $\dbf{A}(G)$, since in that case, one can find out in virtue of $\dcal{T}_G \dmb{e}_i$ which vertices are adjacent to the vertex $v_i$, $i=1,\ldots,n$.

We now can formulate another way of characterizing automorphisms of $G$ through the eigenspaces of $\dcal{T}_G$, which is summarized in the following result.

Let $\dcal{T}$ be a linear operator on $\C^n$. A subspace $U$ of $\C^n$ is said to be $\dcal{T}$-invariant if $\dcal{T}U\subseteq U$.

\begin{thm}\label{Lem-AutomorphismAndOperator} Let $G$ be a graph of order $n$ and let $\sigma$ be a permutation in $\dmf{S}_n$. Then the following statements are equivalent.

\begin{enumerate}

\item[1)] $\sigma$ is an automorphism of $G$.

\item[2)] Every eigenspace of $\dcal{T}_G$ is $\dcal{T}_{\sigma}$-invariant.

\item[3)] $\C^n$ has an orthonormal basis consisting of eigenvectors of both $\dcal{T}_G$ and $\dcal{T}_{\sigma}$.

\item[4)] Every eigenspace of $\dcal{T}_{\sigma}$ is $\dcal{T}_G$-invariant.

\end{enumerate}  \end{thm}

This result not only provides a way to characterize automorphisms of $G$ through the geometric point of view, but more interestingly, presents a way of analyzing eigenvectors of $\dcal{T}_G$. To explain it clearly, we first define two terms. Let $\dmf{G}$ be a permutation group in  $\dmf{S}_n$, and let $\rho:\dmf{G}\rightarrow \GL(\C^n)$ be the {\em permutation representation} of $\dmf{G}$ in $\C^n$, which maps each permutation $\sigma$ in $\dmf{G}$ to the orthogonal operator $\dcal{T}_{\sigma}$. A subspace $U$ of $\C^n$ is said to be {\em $\dmf{G}$-invariant} if $U$ is $\rho(\sigma)$-invariant for every $\sigma$ in $\dmf{G}$.

Since every eigenspace of $\dcal{T}_G$ is $\dcal{T}_{\sigma}$-invariant for every $\sigma$ in the automorphism group $\dmf{G}$ of the grpah $G$, every eigenspace of $\dcal{T}_G$ is $\dmf{G}$-invariant as well, which is equivalent to that for any vector $\dmb{v}$ in an eigenspace $V_{\lambda}$ of $\dcal{T}_G$, $\drm{span}(\dmf{G}\dmb{v})$ is a subspace of  $V_{\lambda}$, where $\drm{span}(\dmf{G}\dmb{v})$ is the subspace in $\C^n$ spanned by the group of vectors $\rho(\sigma)(\dmb{v})=\dcal{T}_{\sigma}\dmb{v}$ obtained by taking all permutations $\sigma$ in $\dmf{G}$. Consequently, $\drm{dim}~V_{\lambda}\ge\drm{dim}~\drm{span}(\dmf{G}\dmb{v})$ for any eigenvector $\dmb{v}$ of $\dcal{T}_G$ corresponding to the eigenvalue $\lambda$.

Naturally, one may wonder for which vector $\dmb{v}$ in $V_{\lambda}$, $\drm{dim}~\drm{span}(\dmf{G}\dmb{v})$ attains the minimum value and for which vector the quantity attains the maximum value, and furthermore, how to determine the exact value of $\drm{dim}~\drm{span}(\dmf{G}\dmb{v})$  for any eigenvector $\dmb{v}$ of $\dcal{T}_G$. We shall answer those questions in the 3rd section by virtue of the linear representation of a finite group.

Lov\'{a}sz \cite{Lovasz} and Baibai \cite{Baibai} established a connection between the eigenvalues of the adjacency matrix of a Cayley graph and irreducible characters of a subgroup of its automorphism group, which inspired us to consider how to use the linear representation of $\drm{Aut}~G$ to analyze eigenspaces of $\dbf{A}(G)$ for a graph $G$ possessing a non-trivial automorphism group.

In the next section, we first formulate two fundamental characterizations to the $\dcal{T}$-invariant subspace for a digonalizable operator $\dcal{T}$ on $\C^n$ and then prove the theorem 1. In the last section, we apply the results established in the 2nd and 3rd sections to the Petersen graph $Pet$ and compute those extremal eigenvectors for $Pet$.

\section{Graph automorphisms viewed as linear operators}

The major goal of this section is to prove the theorem \ref{Lem-AutomorphismAndOperator}, which is stated in the first section. As we have seen, the key of proving that result is to characterize the $\dcal{T}_{\sigma}$-invariant subspace. Since the orthogonal operator $\dcal{T}_{\sigma}$ is diagonalizable, we first present a simple observation about invariant subspaces of a diagonalizable operator.

\begin{lem}\label{Lem-InvariantSubspace} Let $U$ be a subspace of $\C^n$ of dimension $d$, and let $\dcal{T}$ be a diagonalizable operator on $\C^n$. Then the following statements are equivalent.

\begin{enumerate}

 \item[1)] $U$ is $\dcal{T}$-invariant.

 \item[2)] $U$ has a basis consisting of eigenvectors of $\dcal{T}$.

 \item[3)] $U$ has a basis $\dmb{u}_1,\ldots,\dmb{u}_d$ such that $\dcal{T}\dmb{u}_i$ belongs to $U$ where $i=1,\ldots,d$.

\end{enumerate}  \end{lem}

For a given linear operator $\dcal{T}$, it is relatively more difficult to characterize the structure of a subspace $U$ of $\C^n$ in the case that $U$ is {\em not} $\dcal{T}$-invariant. But it is useful to investigate the structure of $U$ since that could reveal the distinction between $\dcal{T}$-invariant and not $\dcal{T}$-invariant subspaces and thus has a particular use in practice as we can see in section 4.

In accordance with Lemma \ref{Lem-InvariantSubspace}, if $\dcal{T}$ be a diagonalizable operator and $U$ is not $\dcal{T}$-invariant then $U$ can be written as direct sum of two subspaces $X$ and $Y$ enjoying that $X$ is $\dcal{T}$-invariant, $Y$ is not trivial, and $\dmb{v}$ is not an eigenvector of $\dcal{T}$ for any vector $\dmb{v}$ in $U\setminus X$. As we shall see below, such a vector $\dmb{v}$ can enable us describe a feature of $U$.

\begin{lem}\label{Lem-NonInvariantSubspace} Let $\dcal{T}$ be a diagonalizable operator, and let $U=X\oplus Y$ be a subspace of $\C^n$ of dimension $d$ such that $X$ is $\dcal{T}$-invariant, $Y$ is not trivial, and $\dmb{v}$ is not an eigenvector of $\dcal{T}$ for any vector $\dmb{v}$ in $U\setminus X$. Then not all vectors $\dcal{T}\dmb{v}, \dcal{T}^2\dmb{v},\ldots,\dcal{T}^d\dmb{v}$ belong to $U$.
\end{lem}

\begin{proof}[\bf Proof]
To prove our assertion, we first establish an auxiliary claim as follows.
\begin{equation}\label{Claim-NonInvariant} \mbox{\em For any vector }\dmb{y}\in Y, \mbox{\em not all vectors }\dcal{T}\dmb{y}, \dcal{T}^2\dmb{y},\ldots,\dcal{T}^d\dmb{y} ~\mbox{\em  belong to }U.  \end{equation}

We show the claim above by a contradiction and assume
\begin{equation}\label{Assumption-Belonging}\mbox{all vectors }\dcal{T}\dmb{y}, \dcal{T}^2\dmb{y},\ldots,\dcal{T}^d\dmb{y}\mbox{ belong to }U.\end{equation}
Since $\drm{\dim}~U=d$, the group of vectors $\dmb{y}, \dcal{T}\dmb{y}, \dcal{T}^2\dmb{y},\ldots,\dcal{T}^d\dmb{y}$ is linear dependent. Consequently, the group must contain a group of linear independent vectors $\dmb{y}, \dcal{T}^{k_1}\dmb{y}, \dcal{T}^{k_2}\dmb{y},\ldots,$ $\dcal{T}^{k_s}\dmb{y}$ $(k_1<k_2<\cdots<k_s)$ so that each vector $\dcal{T}^{i}\dmb{y}$ $(i=0,1,\ldots,d)$ could be written as a linear combination of $\dmb{y}, \dcal{T}^{k_1}\dmb{y},\ldots, \dcal{T}^{k_s}\dmb{y}$, where the number $k_s$ is taken as small as possible.\vspace{2mm}

\noindent{\bf Case 1:} $k_s<d$.

In this case, one can easily check that the subspace $\drm{span}(\dmb{y}, \dcal{T}^{k_1}\dmb{y},\ldots, \dcal{T}^{k_s}\dmb{y})$ is $\dcal{T}$-invariant. Consequently, the subspace possesses, by Lemma \ref{Lem-InvariantSubspace}, a basis consisting of eigenvectors of $\dcal{T}$, and thus $\drm{span}(\dmb{y}, \dcal{T}^{k_1}\dmb{y},\ldots, \dcal{T}^{k_s}\dmb{y})\subseteq X$ according to the relationship between the two subspaces $X$ and $U$, which contradicts the assumption that $\dmb{y}$ is a vector in $Y$ not in $X$.\vspace{2mm}

\noindent{\bf Case 2:} $k_s=d$.

In this case, the vector $\dcal{T}^{d} \dmb{y}$ cannot be, in accordance our choice of the number $k_s$, written as a linear combination of  $\dmb{y}, \dcal{T}\dmb{y}, \dcal{T}^2\dmb{y},$ $\ldots,$  $\dcal{T}^{d-1}\dmb{y}$. Due to our assumption \eqref{Assumption-Belonging},
$$\drm{dim}~\drm{span}(\dmb{y}, \dcal{T}\dmb{y},\dcal{T}^2\dmb{y},\ldots, \dcal{T}^{d-1}\dmb{y})\le d-1.$$
Again, we can check if $\dcal{T}^{d-1} \dmb{y}$ can be written as a linear combination of  $\dmb{y}, \dcal{T}\dmb{y}, \dcal{T}^2\dmb{y},$ $\ldots,$  $\dcal{T}^{d-2}\dmb{y}$.

It is not difficult to see that we can ultimately find, in this way, a vector $\dcal{T}^j \dmb{y}$ $(1 \le j < d)$ satisfying that $\dcal{T}^j \dmb{y}$ can be written as a linear combination of vectors $\dmb{y}, \dcal{T}\dmb{y},\ldots,\dcal{T}^{j-1}\dmb{y}$. Then the group $\dmb{y}, \dcal{T}\dmb{y},\ldots,\dcal{T}^{j-1}\dmb{y}$ contains linear independent vectors $\dmb{y}, \dcal{T}^{l_1}\dmb{y}, \dcal{T}^{l_2}\dmb{y}, \ldots, \dcal{T}^{l_t}\dmb{y}$ $(l_1<l_2<\cdots<l_t<j<d)$ so that all vectors $\dmb{y}, \dcal{T}\dmb{y},\ldots,\dcal{T}^{j}\dmb{y}$ belong to the subspace $\drm{span}(\dmb{y}, \dcal{T}^{l_1}\dmb{y}, \ldots, \dcal{T}^{l_t}\dmb{y})$.

Clearly, the subspace $\drm{span}(\dmb{y}, \dcal{T}^{l_1}\dmb{y}, \ldots, \dcal{T}^{l_t}\dmb{y})$ is $\dcal{T}$-invariant. By means of the same argument in dealing with the first case, $\drm{span}(\dmb{y}, \dcal{T}^{l_1}\dmb{y}, \ldots, \dcal{T}^{l_t}\dmb{y})\subseteq X$, which contradicts the condition that $\dmb{y}$ is a vector in $Y$ not in $X$.\vspace{2mm}

In conclusion, we hold the claim \eqref{Claim-NonInvariant}.\vspace{2mm}

We now prove our assertion that for any vector $\dmb{v}$ in $U\setminus X$, not all vectors $\dcal{T}\dmb{v}, \dcal{T}^2\dmb{v},$ $\ldots,$ $\dcal{T}^d\dmb{v}$ belong to $U$.

Since $U=X\oplus Y$, the vector $\dmb{v}$ can be written uniquely as the sum of $\dmb{x}$ and $\dmb{y}$ where $\dmb{x}\in X$ and $\dmb{y}\in Y$. Since the vector $\dmb{v}$ is not in $X$ and  $Y$ is not trivial, $\dmb{y}$ is not the zero vector. By virtue of the claim \eqref{Claim-NonInvariant}, $\dcal{T}^j\dmb{y}$ does not belong to $U$ for some $j\in\{1,2,\ldots,d\}.$ Due to the assumption that $X$ is $\dcal{T}$-invariant, $\dcal{T}^j\dmb{x}$ is a vector in $X\subseteq U$. Therefore, $\dcal{T}^j\dmb{x}+\dcal{T}^j\dmb{y}=\dcal{T}^j(\dmb{x}+\dmb{y})=\dcal{T}^j\dmb{v}$ is not in $U$, and then our assertion follows.
\end{proof}

A moment's reflection shows that in general case, it is possible that a subspace $U$ of $\C^n$ is not $\dcal{T}$-invariant and $\dcal{T}\dmb{u}, \dcal{T}^2\dmb{u}, \ldots, \dcal{T}^{d-1}\dmb{u}$ belong to $U$ for some vector $\dmb{u}$ in $U$, so the assertion above provides a sharp estimate.\vspace{2mm}

Now we turn to the proof of Theorem \ref{Lem-AutomorphismAndOperator}. We assume that the $n$-dimensional vector space $\C^n$ is endowed with the {\em Hermitian inner product} $\langle\cdot,\cdot\rangle$ such that $\langle\dmb{u},\dmb{v}\rangle=\dmb{v}^*\dmb{u}=\sum_{i=1}^nu_i\overline{v_i}$ for any vectors $\dmb{u}=(u_1,\ldots,u_n)^t$ and $\dmb{v}=(v_1,\ldots,v_n)^t$ in $\C^n$. Two vectors $\dmb{u}$ and $\dmb{v}$ are said to be {\em orthogonal} if $\langle\dmb{u},\dmb{v}\rangle=0$.
A real matrix $\dbf{N}$ is {\em normal} if $\dbf{N}^t\dbf{N}=\dbf{N}\dbf{N}^t$ where $\dbf{N}^t$ denotes the transpose of the matrix $\dbf{N}$. Recall that a permutation matrix is a square matrix enjoying that exactly one entry in each row and column is equal to 1, and all other entries are 0. One can readily check that every permutation matrix is normal.

We are now ready to prove Theorem \ref{Lem-AutomorphismAndOperator}, which presents a geometric characterization of automorphisms of a graph $G$.

\begin{proof}[\bf Proof of Theorem \ref{Lem-AutomorphismAndOperator}]  As we have pointed out before, $\dbf{A}$ and $\dbf{P}_{\sigma}$ are the matrices, respectively, of two linear operators $\dcal{T}_G$ and $\dcal{T}_{\sigma}$ with respect to the standard basis $\dmb{e}_1,\ldots,\dmb{e}_n$. Consequently, $\dcal{T}_G$ and $\dcal{T}_{\sigma}$ can be replaced, respectively, with $\dbf{A}$ and $\dbf{P}_{\sigma}$ in the statements of Theorem \ref{Lem-AutomorphismAndOperator}.

We first show {\it 1)}$\Rightarrow${\it 2)}. Since for every permutation $\sigma\in\dmf{S}_n$, $\sigma$ is an automorphism of $G$ if and only if $\dbf{P}_{\sigma}^t\dbf{A}\dbf{P}_{\sigma}=\dbf{A}$, for any eigenvector $\dmb{v}$ of $\dbf{A}$ associated to some eigenvalue $\lambda$, $$\dbf{P}_{\sigma}^t\dbf{A}\dbf{P}_{\sigma}\dmb{v}=\dbf{A}\dmb{v}=\lambda\dmb{v}.$$
Consequently, $\dbf{A}\dbf{P}_{\sigma}\dmb{v}=\lambda\dbf{P}_{\sigma}\dmb{v}$ for any eigenvector $\dmb{v}$ of $\dbf{A}$, {\it i.e.,} $\dbf{P}_{\sigma}\dmb{v}$ is also an eigenvector of $\dbf{A}$ associated to the eigenvalue $\lambda$, and thus every eigenspace of $\dbf{A}$ is $\dbf{P}_{\sigma}$-invariant.\vspace{2mm}

As we mentioned above, $\dbf{P}_{\sigma}$ is a normal matrix. According to the well-known complex spectral theorem (see \cite{Axler} for details), $\C^n$ has an orthonormal basis consisting of eigenvectors of $\dbf{P}_{\sigma}$, and thus distinct eigenspaces of $\dbf{P}_{\sigma}$ are orthogonal.

If an eigenspace $V_{\lambda}$ of $\dbf{A}$ is $\dbf{P}_{\sigma}$-invariant, $V_{\lambda}$ can be written as direct sum of distinct eigenspaces of $\dbf{P}_{\sigma}$. Therefore, $V_{\lambda}$ has an orthonormal basis consisting of eigenvectors of $\dbf{P}_{\sigma}$.
Since $V_{\lambda}$ is an eigenspace of $\dbf{A}$, that basis consists of eigenvectors of $\dbf{A}$ as well. Clearly, the fact holds true for every eigenspace of $\dbf{A}$. Furthermore,  since $\dbf{A}$ is also normal, distinct eigenspaces of $\dbf{A}$ must be orthogonal in virtue of the complex spectral theorem. Therefore, $\C^n$ has an orthonormal basis consisting of eigenvectors of both $\dbf{A}$ and $\dbf{P}_{\sigma}$, {\it i.e.,} the statement {\it 3)} follows from the statement {\it 2)}.\vspace{2mm}

According to Lemma \ref{Lem-InvariantSubspace}, one can easily see that the statement {\it 3)} implies the statement {\it 4)}. \vspace{2mm}

Finally, we prove that {\it 4)} $\Rightarrow$ {\it 1)}. As we mentioned above, $\sigma\in \drm{Aut}~G$ if and only if $\dbf{P}_{\sigma}^t\dbf{A}\dbf{P}_{\sigma}=\dbf{A},$ which is equivalent to $\dbf{P}_{\sigma}\dbf{A}=\dbf{A}\dbf{P}_{\sigma}$.

Let $\dmb{x}_1,\ldots,\dmb{x}_n$ be an orthonormal basis of $\C^n$, which consists of eigenvectors of $\dbf{P}_{\sigma}$ such that $\dbf{P}_{\sigma}\dmb{x}_i=\eta_i\dmb{x}_i$, $i=1,\ldots,n$. Since every eigenspace of $\dbf{P}_{\sigma}$ is $\dbf{A}$-invariant, for every $\dmb{x}_i$ we have
$$\dbf{P}_{\sigma}\dbf{A}\dmb{x}_i=\eta_i\dbf{A}\dmb{x}_i=\dbf{A}\eta_i\dmb{x}_i
=\dbf{A}\dbf{P}_{\sigma}\dmb{x}_i,~~~i=1,\ldots,n.$$
For an arbitrary vector $\dmb{v}=\sum_{i=1}^na_i\dmb{x}_i$ in $\C^n$,
$$\dbf{P}_{\sigma}\dbf{A}\dmb{v}=\dbf{P}_{\sigma}\dbf{A}\sum_{i=1}^na_i\dmb{x}_i
=\sum_{i=1}^na_i\dbf{P}_{\sigma}\dbf{A}\dmb{x}_i=
\sum_{i=1}^na_i\dbf{A}\dbf{P}_{\sigma}\dmb{x}_i
=\dbf{A}\dbf{P}_{\sigma}\sum_{i=1}^na_i\dmb{x}_i
=\dbf{A}\dbf{P}_{\sigma}\dmb{v}.$$
As a result, $\dbf{P}_{\sigma}\dbf{A}=\dbf{A}\dbf{P}_{\sigma}$, and thus the statement {\it 1)} follows.
\end{proof}

\section{The symmetric and asymmetric eigenvectors}

For brevity, we denote by $\dmf{G}$ the automorphism group $\drm{Aut}~G$ of a graph $G$, which can be, as we have explained in the first section, regarded as a subgroup of $\dmf{S}_n$, provided that $G$ has $n$ vertices. Recall that the permutation representation $\rho$ of $\dmf{G}$ in $\C^n$ we are interested in here is a map from $\dmf{G}$ to $\GL(\C^n)$ such that $\rho(\sigma)=\dcal{T}_{\sigma}$ for every permutation $\sigma\in\dmf{G}$, where $\dcal{T}_{\sigma}$ is the orthogonal operator on $\C^n$ enjoying that $\dcal{T}_{\sigma}\dmb{u} = (u_{\sigma^{-1}1},u_{\sigma^{-1}2},\ldots,u_{\sigma^{-1} n})^t$ for any vector $\dmb{u}=(u_1,\ldots,u_n)^t$ in $\C^n$.

Let $\dbf{A}$ be the adjacency matrix of $G$. Due to Theorem \ref{Lem-AutomorphismAndOperator}, every eigenspace of $\dbf{A}$ is $\dcal{T}_{\sigma}$-invariant, and thus those subspaces are also $\dmf{G}$-invariant. Therefore   the permutation representation $\rho$ of $\dmf{G}$ provides some information about the relationship between the structure of $G$ and eigenvectors of $\dbf{A}$, which will be shown by examples in the next section.

For a given eigenspace $V_{\lambda}$ of $\dbf{A}$, those vectors $\dmb{v}$ in $V_{\lambda}$, enjoying that $\drm{dim}~\drm{span}(\dmf{G}\dmb{v})$ can attain the minimum value, are called {\em symmetric vectors} of $V_{\lambda}$, while those vectors $\dmb{v}$, enjoying that $\drm{dim}~\drm{span}(\dmf{G}\dmb{v})$ can attain the maximum value, are called {\em asymmetric vectors} of $V_{\lambda}$. In this section, we characterize symmetric vectors and the asymmetric vectors of $V_{\lambda}$, respectively, and furthermore, we determine the exact value of $\drm{dim}~\drm{span}(\dmf{G}\dmb{v})$ for any eigenvector $\dmb{v}$ of $\dbf{A}$.\vspace{2mm}

Let $\rho$ and $\rho'$ be two representations of $\dmf{G}$ in vector spaces $V$ and $V'$. These representations $\rho$ and $\rho'$ are said to be {\em isomorphic} if there exists a linear isomorphism $\phi:V \rightarrow V'$ which ``transform'' $\rho$ into $\rho'$, {\it i.e.,}
$$\phi\circ\rho(\sigma)=\rho'(\sigma)\circ\phi,~~~~~\forall \sigma\in\dmf{G}.$$
Since $\phi$ is an isomorphism between $V$ and $V'$, $\drm{dim}~V=\drm{dim}~V'$. Let the dimension of $V$ be equal to $m$. Then one can readily see that for every $\sigma\in\dmf{G}$,  the matrices of $\rho(\sigma)$ and $\rho'(\sigma)$ are the same with respect to the basses $\dmb{b}_1, \ldots, \dmb{b}_m$ in $V$ and $\phi(\dmb{b}_1), \ldots, \phi(\dmb{b}_m)$ in $V'$, respectively. The isomorphism between $\rho$ and $\rho'$ indicates that the two representations have a close relationship, which makes, as we shall see in what follows, some trouble in characterizing those asymmetric eigenvectors.

Let $W$ be a non-trivial subspace in $\C^n$. The representation $\rho$ of $\dmf{G}$ in $W$ is said to be {\em irreducible} if $W$ is $\dmf{G}$-invariant and there is no non-trivial subspace of $W$ which is also $\dmf{G}$-invariant. It is not difficult to see that the subspace $\drm{span}(\dmf{G}\dmb{v})$ is $\dmf{G}$-invariant for any vector $\dmb{v}$ in $V_{\lambda}$, so $\drm{span}(\dmf{G}\dmb{v})$ can be decomposed as a direct sum of irreducible representations of $\rho$. On the other hand, if $W$ is an irreducible representation of $\rho$, then for any non-trivial vector $\dmb{w} \in W$, $\drm{span}(\dmf{G}\dmb{w})=W$ in accordance with the definition of the irreducible representation. Therefore, symmetric vectors of $V_{\lambda}$ are those which belong to some irreducible representation $W$ such that $\drm{dim}~W \le \drm{dim}~U$ for any $\dmf{G}$-invariant and non-trivial subspace $U$ of $V_{\lambda}$. \vspace{2mm}

It is difficult, however, to characterize those asymmetric eigenvectors of $\dbf{A}$. Before turning to that tough part in this section, we first present a simple observation which enables us to formulate an upper bound of the dimension of the subspace $\spa(\dmf{G} \dmb{v})$. Set
$$\dmf{G}[\dmb{v}]=\{\sigma\in \dmf{G} : \dmb{v} \text{ and } \rho(\sigma) \dmb{v} \text{ are linear dependent}\}.$$
Then one can readily verify that $\dmf{G}[\dmb{v}]$ is a subgroup of $\dmf{G}$, for if $\sigma_1,\sigma_2\in \dmf{G}[\dmb{v}]$ then $\sigma_1\sigma_2\in \dmf{G}[\dmb{v}]$. Furthermore, it is plain to see that $\rho(\alpha) \dmb{v}$ and $\rho(\beta) \dmb{v}$ are linear independent if $\beta^{-1}\alpha\notin \dmf{G}[\dmb{v}]$, so the following is relevant.

\begin{lem}\label{Lem-SimpleObservation}
$$\mathrm{dim}~\mathrm{span}(\dmf{G} \dmb{v})\leq\dugi{\dmf{G}}{\dmb{v}},$$ and the equality holds if and only if
$$\mathrm{span}(\dmf{G} \dmb{v})=\bigoplus_{i=1}^m \mathrm{span}(\rho(\sigma_i)\dmb{v})\Rightarrow \dmf{G} = \bigcup_{i=1}^m\sigma_i \dmf{G}[\dmb{v}],$$ where $\sigma_i\in \dmf{G}$ and $m=\dugi{\dmf{G}}{\dmb{v}}$.
\end{lem}

One can readily see that the upper bound above is not tight in general, for $\dugi{\dmf{G}}{\dmb{v}}$ indicates the number of vectors in the group of vectors $\dmf{G}\dmb{v}$ which are pairwise independent, while $\mathrm{dim}~\mathrm{span}(\dmf{G} \dmb{v})$ is the number of vectors that a maximum linearly independent group in $\dmf{G}\dmb{v}$ can have. However, the inequality does connect two distinct objects which look irrelevant, so it seems to deserve some mention. \vspace{2mm}

Now we turn to asymmetric eigenvectors of $\dbf{A}$. As we have explained before, a symmetric vector of a eigenspace $V_{\lambda}$ must be in one irreducible representation of $\rho$ in $V_{\lambda}$ whose dimension is the smallest one among all irreducible representations of $\rho$ in $V_{\lambda}$. Accordingly, one may reasonably expect that
$$\mbox{an asymmetric vector of }V_{\lambda}\mbox{ would have the form } \dmb{w}_1+\dmb{w}_2+\cdots+\dmb{w}_k,$$
where $\dmb{w}_i\in W_i$ $(i=1,\ldots,k)$ and $\oplus_{i=1}^k W_i=V_{\lambda}$ is the decomposition into irreducible representations of  $\rho$ in $V_{\lambda}$. Fundamentally, it is right, but generally, it is not.

The trouble is that if there are some irreducible representations of $\rho$ in $V_{\lambda}$ which are isomorphic to one another, then the decomposition $\oplus_{i=1}^k W_i$ is not unique, as Lemma \ref{LemIsomorphism} illustrates below, so we cannot, in that case, arbitrarily choose a vector in $W_i$ and obtain an asymmetric vector of $V_{\lambda}$ in the trivial way above. Theorem \ref{Thm-2ndCaseForEigenvector} presents the key requirement that those vectors $\dmb{w}_i$ should enjoy, to obtain an asymmetric vector in the form $\sum_{i=1}^k \dmb{w}_i$.

However, if any two irreducible representation in the decomposition $\oplus_{i=1}^k W_i$ are not isomorphic, then $\sum_{i=1}^k \dmb{w}_i$ is indeed an asymmetric vector of $V_{\lambda}$, as shown below by Theorem \ref{Thm-1stCaseForEigenvector}. In that case, $\drm{span}(\dmf{G}(\sum_{i=1}^k \dmb{w}_i)) = V_{\lambda}$.

Before we go further, we make a comment that the proofs of those results in what follows work not only for the permutation representation we consider here but also for any linear representation of $\dmf{G}$ in $\C^n$. Since every eigenspace $V_{\lambda}$ of $\dbf{A}$ is $\dmf{G}$-invariant, it is not necessary to distinguish distinct eigenspaces of $\dbf{A}$ in stating our assertions, and thus we uniformly formulate those assertions for the whole space $\C^n$.

\begin{lem}\label{LemIsomorphism} If $\phi$ is an isomorphism between two irreducible representations $W_1$ and $W_2$ of $\rho$ in $\C^n$, then $\spa(\dmf{G}[\dmb{w}_1+\phi(\dmb{w}_1)])$ is isomorphic to $W_1$ where $\dmb{w}_1$ is a non-trivial vector in $W_1$.
\end{lem}

\begin{proof}[\bf Proof] Since $W_1$ is an irreducible representation of $\rho$, one can choose some elements $\sig_1=1_\dmf{G},\ldots,\sig_d$ of $\dmf{G}$ such that $\rho(\sig_1)(\dmb{w}_1)=\dmb{w}_1,\ldots,\rho(\sig_d)(\dmb{w}_1)$ form a basis of $W_1$. It is plain to verify that $\rho(\sig_1)(\dmb{w}_1)+\rho(\sig_1)\phi(\dmb{w}_1),\ldots,\rho(\sig_d)(\dmb{w}_1)+\rho(\sig_d)\phi(\dmb{w}_1)$ form a basis of the subspace $\spa(\dmf{G}[\dmb{w}_1+\phi(\dmb{w}_1)])$.

We further define a linear map
$$\psi:W_1\rightarrow \spa(\dmf{G}[\dmb{w}_1+\phi(\dmb{w}_1)])$$ satisfying
$$\psi:\rho(\sig_i)(\dmb{w}_1)\mapsto\rho(\sig_i)(\dmb{w}_1)+\rho(\sig_i)\phi(\dmb{w}_1), 1\le i\le d.$$
Clearly, $\psi$ is an isomorphism between the two subspaces $W_1$ and $\spa(\dmf{G}[\dmb{w}_1+\phi(\dmb{w}_1)])$. Moreover, it is also routine to check that $\psi$ is an isomorphism between the two representations $W_1$ and $\spa(\dmf{G}[\dmb{w}_1+\phi(\dmb{w}_1)])$. In fact, for any vector $\dmb{w}$ in $W_1$, $\dmb{w}=\sum_{i=1}^dx_i\rho(\sig_i)(\dmb{w}_1)$ where $x_i\in \C$ ($i=1,\ldots,d$), for $\rho(\sig_1)(\dmb{w}_1)=\dmb{w}_1,\ldots,\rho(\sig_d)(\dmb{w}_1)$ form a basis of $W_1$. Picking arbitrarily an element $\tau\in \dmf{G}$, we assume $\rho(\tau)(\dmb{w})=\sum_{i=1}^dy_i\rho(\sig_i)(\dmb{w}_1)$ where $y_i\in \C$ ($i=1,\ldots,d$). Consequently,
\begin{align*}\psi\rho(\tau)(\dmb{w}) & = \psi\left[\sum_{i=1}^dy_i\rho(\sig_i)(\dmb{w}_1)\right]\\
& = \sum_{i=1}^dy_i\psi\rho(\sig_i)(\dmb{w}_1)\\
& = \sum_{i=1}^d y_i\left(\rho(\sig_i)(\dmb{w}_1)+\rho(\sig_i)\phi(\dmb{w}_1)\right).\end{align*}
On the other hand,
\begin{align*}\rho(\tau)\psi(\dmb{w}) & = \rho(\tau)\psi\left[\sum_{i=1}^d    x_i\rho(\sig_i)(\dmb{w}_1)\right]\\
& = \rho(\tau)\left[\sum_{i=1}^d x_i\psi\rho(\sig_i)(\dmb{w}_1)\right]\\
& = \rho(\tau)\left[\sum_{i=1}^d x_i\left(\rho(\sig_i)(\dmb{w}_1)+\rho(\sig_i)\phi(\dmb{w}_1)\right)\right]\\
& = \rho(\tau)\left[\sum_{i=1}^d x_i\rho(\sig_i)(\dmb{w}_1)\right] + \rho(\tau)\left[\sum_{i=1}^d x_i\rho(\sig_i)\phi(\dmb{w}_1)\right] \\
& = \sum_{i=1}^d y_i\rho(\sig_i)(\dmb{w}_1) + \phi\rho(\tau)\left[\sum_{i=1}^d x_i\rho(\sig_i)(\dmb{w}_1)\right]\\
& = \sum_{i=1}^d y_i\rho(\sig_i)(\dmb{w}_1) + \sum_{i=1}^d y_i\rho(\sig_i)\phi(\dmb{w}_1).\end{align*}
Consequently, $\psi\rho(\tau)=\rho(\tau)\psi$ for any $\tau\in \dmf{G}$, and thus $\psi$ is an isomorphism between the two representations $W_1$ and $\spa(\dmf{G}[\dmb{w}_1+\phi(\dmb{w}_1)])$.
\end{proof}

We are now ready to show that $\sum_{s=1}^k \dmb{w}_s$ is an asymmetric vector in the relatively simple case that any two irreducible representations in the decomposition of $\C^n$ into irreducible representations of $\rho$ are not isomorphic.

\begin{thm}\label{Thm-1stCaseForEigenvector} Suppose $\C^n=\oplus_{s=1}^kW_s$ is a decomposition into irreducible representations of $\rho$. Then three assertions below are equivalent.

\begin{enumerate}

\item[1)] For any two irreducible representations $W_i$ and $W_j$ $(i\neq j)$, $W_i$ is not isomorphic to $W_j$.

\item[2)] If $U$ is a $\dmf{G}$-invariant subspace of $\C^n$ under the representation $\rho$ with a decomposition $\oplus_{t=1}^l X_{t} = U$ into irreducible representations, then the decomposition is unique, {\it i.e.,} for each $X_{t}$ $(t=1,\ldots,l)$ there exists uniquely an irreducible representation $W_s$ $(1\le s \le k)$ in the decomposition of $\C^n$ such that $X_{t}= W_s$.

\item[3)] $\spa(\dmf{G}(\sum_{s=1}^k \dmb{w}_s))=\C^n$ where $\dmb{w}_s$ is a non-trivial vector in $W_s$ $(s=1,\ldots,k)$.

\end{enumerate}
\end{thm}
\begin{proof}[\bf Proof] We first show that $\it 1)\Rightarrow 2)$. Clearly, it suffices to show that if $X_t\cong W_s$ then $X_t=W_s$.

Let $p_i$ be the projection of $\C^n$ onto the irreducible representation $W_i$, $i=1,\ldots,k$. Set $$p_i^0=\frac{1}{|\dmf{G}|}\sum_{\sig\in \dmf{G}}\rho(\sig)p_i\rho(\sig^{-1}), ~~~ i=1,\ldots,k.$$
Since $W_i$ is $\dmf{G}$-invariant, we have, for any $\dmb{w}_i\in W_i$ and $\sig\in \dmf{G}$,
$$p_i\rho(\sig^{-1})(\dmb{w}_i)=\rho(\sig^{-1})(\dmb{w}_i)\Rightarrow \rho(\sig)p_i\rho(\sig^{-1})(\dmb{w}_i)=\dmb{w}_i\Rightarrow p_i^0(\dmb{w}_i)=\dmb{w}_i.$$
For another irreducible representation $W_j$ of $\C^n$, one can readily see, by the same reason, that $p_i^0\dmb{w}_j=0$ for every vector $\dmb{w}_j\in W_j$. Consequently, $p_i^0$ is also a projection of $\C^n$ onto $W_i$. Furthermore, for any $\tau\in \dmf{G}$, we have, according to the definition of $p_i^0$,
$$\rho(\tau)p_i^0\rho(\tau^{-1})=\frac{1}{|\dmf{G}|}\sum_{\sig\in \dmf{G}}\rho(\tau)\rho(\sig)p_i\rho(\sig^{-1})\rho(\tau^{-1})= \frac{1}{|\dmf{G}|}\sum_{\sig\in \dmf{G}}\rho(\tau\sig)p_i\rho[(\tau\sig)^{-1}]=p_i^0. $$
Thus, $\rho(\tau)p_i^0=p_i^0\rho(\tau)$ for any $\tau\in \dmf{G}$ and $p_i^0$ $(i=1,\ldots,k)$.

We now consider the restriction of $p_i^0$ to $X_t$ denoted by $p_i^0|_{X_t}$, which is a projection of $X_t$ onto $W_i$. Further, noting the fact that $p_i^0$ and $\rho(\tau)$ are commutative for any $\tau\in \dmf{G}$ and $p_i^0$ $(i=1,\ldots,k)$, one can readily see that   $\rho(\tau)p_i^0|_{X_t}=p_i^0|_{X_t}\rho(\tau)$ for any $\tau\in \dmf{G}$ and $p_i^0|_{X_t}$.

According to the assumption that $W_s\ncong W_i$ for any irreducible representation $W_i$ $(i\neq s)$ of $\C^n$, it follows from Schur's Lemma that $p_i^0|_{X_t}=0$. By the fact that $p_i^0|_{X_t}$ $(i=1,\ldots,k)$ is a projection of $X_t$ onto $W_i$, $\sum_{i=1}^k p_i^0|_{X_t}$ is the identity map. Further, since $X_t\cong W_s$,  $p_s^0|_{X_t}:X_t\rightarrow W_s$ is an isomorphism, and therefore, $p_s^0|_{X_t}$ is an identity map. Consequently,  $X_t=W_s$.\vspace{2mm}

Next, we show that $\it 2)\Rightarrow 3)$. Set $\dmb{v}=\sum_{s=1}^k \dmb{w}_s$ for brevity. We assume, for a contradiction, that \begin{equation}\label{Assumption-1}\spa(\dmf{G}\dmb{v})=U\subsetneq \C^n.\end{equation}
Clearly, $U$ is a $\dmf{G}$-invariant subspace of $\C^n$ under the representation $\rho$. By means of the statement {\it 2)}, for any irreducible representation $X_t$ of $U$, $X_t$ is also a member of the decomposition of $\C^n$ into irreducible representations. According to the assumption \eqref{Assumption-1}, some irreducible representation of $\C^n$, however, is not a member of the decomposition $\oplus_{t=1}^l X_{t} = U$.  In that case, the vector $\dmb{v}=\sum_{s=1}^k \dmb{w}_s$, which is equal to $\rho(1_{\dmf{G}})(\dmb{v})$ and thus belongs   to $\spa(\dmf{G}\dmb{v})$, does not belong to $U$, that is a contradiction.\vspace{2mm}

We finally show that $\it 3)\Rightarrow 1)$ by a contradiction. To hold the desire, we assume, without losing any generality, that $W_1$ and $W_2$ are isomorphic. Let $\phi$ be an isomorphism between the two irreducible representations $W_1$ and $W_2$.

We now construct a vector $$\dmb{u}=\dmb{w}_1+\phi(\dmb{w}_1)+\sum_{s=3}^k \dmb{w}_s$$ where $\dmb{w}_s$ is a non-trivial vector in $W_s$ ($1\le s\le k$ and $s\neq 2$). According to Lemma \ref{LemIsomorphism}, we have $\spa(\dmf{G}[\dmb{w}_1+\phi(\dmb{w}_1)])\cong W_1$ and thus
$$\spa(\dmf{G}[\dmb{w}_1+\phi(\dmb{w}_1)])\bigoplus
\spa\left(\dmf{G}\left[\sum_{s=3}^k \dmb{w}_s\right]\right)\subsetneq\C^n.$$
By means of a plain fact that
$$\spa(\dmf{G}\dmb{u})\subseteq\spa(\dmf{G}[\dmb{w}_1+\phi(\dmb{w}_1)])\bigoplus
\spa\left(\dmf{G}\left[\sum_{s=3}^k \dmb{w}_s\right]\right),$$
we therefore obtain that
$$\spa(\dmf{G}\dmb{u})\subsetneq\C^n,$$
which contradicts the statement {\it 3)}.
\end{proof}

We now turn to the harder case that in a decomposition of $\C^n$ into irreducible representations of $\rho$, there are some of irreducible representations which are isomorphic to one another. As shown in Lemma \ref{LemIsomorphism}, the way of decomposing $\C^n$ into irreducible representations is not unique in this case. But there is a canonical way of organizing those irreducible representations.

Let $W_1,\ldots,W_h$ be the all distinct irreducible representations of $\rho$ in $\C^n$ such that any two representations in that group are not isomorphic. The decomposition  $\C^n=\oplus_{i=1}^h V_i$ into subrepresentations of $\rho$ is said to be {\em canonical decomposition of $\rho$} if $V_i=\oplus_{j=1}^{m_i} W_{ij}$ ($i=1,\ldots,h$) is a decomposition into irreducible representations of $\rho$ such that each irreducible representation $W_{ij}$ $(j=1,\ldots,m_i)$ is isomorphic to $W_i$.

Since any two irreducible representations $W_i$ and $W_{i'}$ are not isomorphic if $i\neq i'$, one can readily see, by virtue of the approach in proving Theorem \ref{Thm-1stCaseForEigenvector}, that the decomposition $\oplus_{i=1}^h V_i=\C^n$ is unique. As a result, $\drm{span}(\dmf{G}(\sum_{i=1}^h \dmb{v}_i)) = \C^n$ provided that $\drm{span}(\dmf{G} \dmb{v}_i) = V_i$ where the vector $\dmb{v}_i$ belongs to $V_i$ ($i=1,\ldots,h$). The decomposition $\oplus_{j=1}^{m_i} W_{ij}$ of $V_i$, however, is not unique, as shown in Lemma \ref{LemIsomorphism}. Consequently, to hold the desire that $\drm{span}(\dmf{G}(\sum_{j=1}^{m_i} \dmb{w}_{ij})) = V_i$, the group of vectors $\dmb{w}_{i1},\ldots,\dmb{w}_{im_i}$ need to comply with a further requirement, where the vector $\dmb{w}_{ij}$ belongs to $W_{ij}$ ($j=1,\ldots,m_i$).  We first present a lemma below for introducing the key requirement.

Let $U_i$ ($i=1,2$) be a subspace of $\C^n$, and let $\rho_i:\dmf{G}\rightarrow \GL(U_i)$ be a linear representation. We define a vector space
$$\Hom_\dmf{G}(U_1,U_2)=\{\phi:U_1\rightarrow U_2\mid \phi\mbox{ is a linear map and }
\phi\rho_1(\sig)=\rho_2(\sig)\phi, \forall\sig\in \dmf{G}\}.$$

Let $W$ be a subspace of $\C^n$ such that the linear representation $\rho_{_W}:\dmf{G}\rightarrow \GL(W)$ is isomorphic to some irreducible representation $W_i$ of $\rho$. Then $W$ is an irreducible representation, and thus, for each $\phi\in\Hom_\dmf{G}(W,V_i)$, $\phi$ vanishes or is injective.

\begin{lem}\label{Lem-Dimension} For any fixed decomposition $V_i=\oplus_{j=1}^{m_i}W_{ij}$ into irreducible representations isomorphic to $W_i$, if the maps $\phi_1,\ldots,\phi_{m_i}$ in $\Hom_\dmf{G}(W,V_i)$ satisfy that $\mbox{\rm Im }\phi_j=W_{ij}$ $(j=1,\ldots,m_i)$, then $\phi_1,\ldots,\phi_{m_i}$ form a basis of $\Hom_\dmf{G}(W,V_i)$.
\end{lem}

\begin{proof}[\bf Proof]
We first show that $\phi_1,\ldots,\phi_{m_i}$ are linear independent. If it is not the case, then there exist not all zeros $k_1,\ldots,k_{m_i}$ in $\C$ such that $\sum_{j=1}^{m_i}k_j\phi_j=0$. Without losing any generality, we suppose $k_1\neq 0$, and thus obtain $\phi_1=-\sum_{j=2}^{m_i}k_1^{-1}k_j\phi_j$ which contradicts the requirement that $\mbox{Im }\phi_1=W_{i1}$.

We now show that for each $\phi\in\Hom_\dmf{G}(W,V_i)$, one can find out some complex numbers $\lambda_1,\ldots,\lambda_{m_i}$ so that $\phi=\sum_{j=1}^{m_i}\lambda_j\phi_j$. As we have seen in the proof of Lemma \ref{LemIsomorphism}, one can construct a projection $p_j^0$ of $V_i$ onto $W_{ij}$ ($1\le j\le m_i$) such that $\rho(\sig)p_j^0=p_j^0\rho(\sig)$ for any $\sig\in \dmf{G}$. Evidently, $$\phi=\sum_{j=1}^{m_i}p_j^0\phi.$$
Let $\psi_j:W\rightarrow W_{ij}$ be a linear map such that $\psi_j(\dmb{w})=\phi_j(\dmb{w})$ ($1\le j\le m_i$) for every $\dmb{w}\in W$. Since $\phi_j$ is an injection, $\psi_j$ is an bijection and thus $\psi_j^{-1}p_j^0\phi:W\rightarrow W$ ($1\le j\le m_i$). Accordingly, one can readily see that $$\rho(\sig)\psi_j^{-1}p_j^0\phi\rho(\sig^{-1})=\psi_j^{-1}p_j^0\phi
\mbox{ for any }\sig\in \dmf{G},$$
so there exists some $\lambda_j\in\C$ such that $\psi_j^{-1}p_j^0\phi=\lambda_jI$ ($1\le j\le m_i$) by Schur's lemma. Consequently, by the definition of $\psi_j$, we have $$p_j^0\phi=\lambda_j\psi_j=\lambda_j\phi_j, 1\le j\le m_i.$$
Therefore, $\phi=\sum_{j=1}^{m_i}\lambda_j\phi_j$.
\end{proof}

\begin{cor}\label{Cor-Dimension} $\dimn \Hom_\dmf{G}(W,V_i)=\dimn V_i/\dimn W$.
\end{cor}

Let $\C^n=\oplus_{i=1}^hV_i$ be the canonical decomposition of the representation $\rho$, and let $V_i=\oplus_{j=1}^{m_i}W_{ij}$ be a decomposition into irreducible representations isomorphic to $W_i$. A group of vectors $\dmb{w}_1,\ldots,\dmb{w}_{m_i}$ is said to be {\em an independent group} if the vectors $\psi_{1m_i}(\dmb{w}_1),\psi_{2m_i}(\dmb{w}_2),$ $\ldots,$ $\psi_{m_{i}-1m_i}(\dmb{w}_{m_{i}-1}), \dmb{w}_{m_i}$ are linear independent where $\dmb{w}_j$ is a non-trivial vector in $W_{ij}$ ($j=1,\ldots,m_i$) and $\psi_{xy}\in\Hom_\dmf{G}(W_{ix},W_{iy})$. By means of Corollary \ref{Cor-Dimension}, $\dimn \Hom_\dmf{G}(W_{ix},W_{iy})$ $=$ $1$ $(1\le x,y\le m_i)$ and thus the independent group is well defined. One can easily see that if  $\psi_{1m_i}(\dmb{w}_1),\ldots,\psi_{m_{i}-1m_i}(\dmb{w}_{m_{i}-1}), \dmb{w}_{m_i}$ are linear  independent, then each linear map $\psi_{jm_i}$ is bijective ($j=1,\ldots,m_i$).

\begin{thm}\label{Thm-2ndCaseForEigenvector} Let $V_i$ be one of terms in the canonical decomposition of $\rho$. Then the three statements below are equivalent.

\begin{enumerate}

\item[1)] There exists a vector $\dmb{v}=\sum_{j=1}^{m_i}\dmb{w}_j$ in $V_i$ enjoying that $\spa(\dmf{G}\dmb{v})=V_i$ where $\dmb{w}_j$ is a non-trivial vector in $W_{ij}$ $(j=1,\ldots,m_i)$.

\item[2)] The group of vectors $\dmb{w}_1,\ldots,\dmb{w}_{m_i}$ is an independent group.

\item[3)] $(\dimn W_i)^2\ge \dimn V_i$, which is equivalent to that $\dimn W_i \ge m_i$.

\end{enumerate}
\end{thm}

\begin{remark} In the case that $(\dimn W_i)^2<\dimn V_i$,
$\dimn \spa(\dmf{G}\dmb{u})\le(\dimn W_i)^2$ for any vector $\dmb{u}$ in $V_i$ and one can always find a vector $\dmb{v}$ with $\dimn \spa(\dmf{G}\dmb{v})=(\dimn W_i)^2$ according to the proof of our theorem.
\end{remark}

\begin{proof}[\bf Proof]

We first show the equivalence between the first two statements. To prove {\it 1)$\Rightarrow$2)}, we assume for a contradiction that $\psi_{1m_i}(\dmb{w}_1),\ldots,\psi_{lm_i}(\dmb{w}_l)$ $(l\le m_i-1)$ are linear independent, but $\psi_{1m_i}(\dmb{w}_1),\ldots,\psi_{lm_i}(\dmb{w}_l),\psi_{l+1m_i}(\dmb{w}_{l+1})$ are linear dependent where $\psi_{xy}\in\Hom_\dmf{G}(W_{ix},W_{iy})$. Consequently, there exist not all zeros $k_1,\ldots,k_l$ in $\C$ so that
$$\psi_{l+1m_i}(\dmb{w}_{l+1})=\sum_{j=1}^l k_j\psi_{jm_i}(\dmb{w}_j)$$
and thus
$$\dmb{w}_{l+1}=\sum_{j=1}^l k_j\psi_{l+1m_i}^{-1}\psi_{jm_i}(\dmb{w}_j).$$
Therefore,
$$\dmb{v}=\sum_{j=1}^{m_i}\dmb{w}_j=\sum_{j=1}^{l}\big(\dmb{w}_j+k_j\psi_{l+1m_i}^{-1}\psi_{jm_i}(\dmb{w}_j)\big)
+\sum_{j=l+2}^{m_i}\dmb{w}_j.$$
Apparently, $\psi_{l+1m_i}^{-1}\psi_{jm_i}\in\Hom_\dmf{G}(W_{ij},W_{il+1})$. By means of Lemma \ref{LemIsomorphism}, the two representations $\spa(\dmf{G}[\dmb{w}_j+k_j\psi_{l+1m_i}^{-1}\psi_{jm_i}(\dmb{w}_j)])$ and $W_i$ are isomorphic, $1\le j\le l$. Consequently,
$$\bigoplus_{j=1}^l\spa(\dmf{G}[\dmb{w}_j+k_j\psi_{l+1m_i}^{-1}\psi_{jm_i}(\dmb{w}_j)])
\bigoplus_{j=l+2}^{m_i}\spa(\dmf{G} \dmb{w}_j)\cong V_i\setminus W_{il+1}.$$
According a plain fact that
$$\spa(\dmf{G} \dmb{v})\subseteq\bigoplus_{j=1}^l\spa(\dmf{G}[\dmb{w}_j+k_j\psi_{l+1m_i}^{-1}\psi_{jm_i}(\dmb{w}_j)])
\bigoplus_{j=l+2}^{m_i}\spa(\dmf{G} \dmb{w}_j),$$
we have
$\spa(\dmf{G} \dmb{v})\subsetneq V_i$ which contradicts the assumption enjoyed by the vector $\dmb{v}$.

Next, we show that {\it 2)$\Rightarrow$1)} by a contradiction and assume that $\spa(\dmf{G} \dmb{v})\subsetneq V_i$. According to the uniqueness of the canonical decomposition (see $\S 2.6$ in \cite{Serre}, for example), $\spa(\dmf{G} \dmb{v})$ has a decomposition $\spa(\dmf{G} \dmb{v})=\oplus_{s=1}^t X_s$ into irreducible representations such that $X_s\cong W_i$, $1\le s\le t$. By our assumption above, we have
\begin{equation}\label{Equ-Contradiction} t<m_i.\end{equation}

Let $\phi_j$ be a map in $\Hom_\dmf{G}(W_{im_i},V_i)$ such that $\mbox{Im }\phi_j=W_{ij}$ $(1\le j\le m_i)$, and let $\vp_s$ be a map in $\Hom_\dmf{G}(W_{im_i},V_i)$ such that $\mbox{Im }\vp_s=X_{s}$ $(1\le s\le t)$. Accordingly, $\phi_j$ and $\vp_s$ are injective, $1\le j\le m_i$ and $1\le s\le t$. Furthermore, it would be apparent on a moment consideration that for each $\phi_j$ $(1\le j\le m_i)$, there exists  $\psi_{m_ij}\in\Hom_\dmf{G}(W_{im_i},W_{ij})$ such that $\phi_j(\dmb{w})=\psi_{m_ij}(\dmb{w})$ for any $\dmb{w}\in W_{im_i}$.

Applying Lemma \ref{Lem-Dimension}, $\phi_1,\ldots,\phi_{m_i}$ form a basis of $\Hom_\dmf{G}(W_{im_i},V_i)$ and therefore for each $\vp_s$ $(1\le s\le t)$, there exist not all zeros $k_{s1},\ldots,k_{sm_i}$ in $\C$  such that
\begin{equation}\label{Equ-PhiBasis} \vp_s=\sum_{j=1}^{m_i} k_{sj}\phi_j.\end{equation}
Evidently, $\dmb{v}=\sum_{j=1}^{m_i} \dmb{w}_j$ belongs to $\spa(\dmf{G} \dmb{v})$. Since $\spa(\dmf{G} \dmb{v})=\oplus_{s=1}^t X_s$, one can pick a vector $\dmb{x}_s$ in $X_s$ $(1\le s\le t)$ so that
\begin{equation}\label{Equ-Coordinate} \sum_{j=1}^{m_i} \dmb{w}_j=\sum_{s=1}^t \dmb{x}_s. \end{equation}
Let $\dmb{z}_s$ be the vector in $W_{m_i}$ such that $\vp_s(\dmb{z}_s)=\dmb{x}_s$ $(1\le s\le t)$. By means of equations \eqref{Equ-PhiBasis} and \eqref{Equ-Coordinate}, we have
\begin{align*}
\sum_{j=1}^{m_i} \dmb{w}_j &= \sum_{s=1}^t \vp_s(\dmb{z}_s)\\
&=\sum_{s=1}^t\left(\sum_{j=1}^{m_i} k_{sj}\phi_j\right)(\dmb{z}_s)\\
&=\sum_{j=1}^{m_i}\phi_j\left(\sum_{s=1}^t k_{sj}\dmb{z}_s\right)\\
&=\sum_{j=1}^{m_i}\psi_{m_ij}\left(\sum_{s=1}^t k_{sj}\dmb{z}_s\right).
\end{align*}
Consequently,
$$\dmb{w}_j=\psi_{m_ij}\left(\sum_{s=1}^t k_{sj}\dmb{z}_s\right), j=1,2,\ldots,m_i.$$
Noting the fact that $\psi_{m_ij}$ is bijective, we have
$$\psi_{m_ij}^{-1}(\dmb{w}_j)=\sum_{s=1}^t k_{sj}\dmb{z}_s, j=1,2,\ldots,m_i.$$
Hence, the group of vectors $\psi_{m_i1}^{-1}(\dmb{w}_1),\ldots,\psi_{m_im_i-1}^{-1}(\dmb{w}_{m_i-1}), \psi_{m_im_i}^{-1}(\dmb{w}_{m_i})=\lambda_{m_i}\dmb{w}_{m_i}$ in $W_{m_i}$ could be expressed linearly by the group $\dmb{z}_1,\ldots,\dmb{z}_t$ in $W_{m_i}$. However, the fact that $\dmb{w}_1,\ldots,\dmb{w}_{m_i}$ are independent contradicts the claim \eqref{Equ-Contradiction}, {\it i.e.,} $t<m_i$, which is a conclusion of our assumption that $\spa(\dmf{G} \dmb{v})\subsetneq V_i$.

\vspace{2mm}

We now show the equivalence between the last two statements. According to the definition, if the group of vectors $\dmb{w}_1,\ldots,\dmb{w}_{m_i}$ is an independent group then $\psi_{1m_i}(\dmb{w}_1),\ldots,$ $\psi_{m_{i}-1m_i}(\dmb{w}_{m_{i}-1}),$ $\dmb{w}_{m_i}$ are linear independent where $\psi_{xy}\in\Hom_\dmf{G}(W_{ix},W_{iy})$. Evidently,  $\psi_{1m_i}(\dmb{w}_1),\ldots,\psi_{m_{i}-1m_i}(\dmb{w}_{m_{i}-1}),$ $\dmb{w}_{m_i}$ are in the subspace $W_{m_i}$, and thus $$\dimn W_{m_i}\ge m_i=\dimn V_i/\dimn W_{m_i}.$$ Since $W_{m_i}\cong W_i$, $(\dimn W_i)^2\ge \dimn V_i$.

On the other hand, one can readily show that \textit{3)$\Rightarrow$2)} by induction. To be precise, we shall show by induction that if $\dimn W_{i}\ge\dimn V_i/\dimn W_{i}$ then there exists an independent group $V_i$.

Clearly, the assertion holds for the trivial case that $\dimn V_i/\dimn W_i=1$. In fact, every nonzero vector in $W_{i1}$ could be an independent group in that case.

We assume that our assertion holds in the case that $\dimn V_i/\dimn W_i=m$. For the case that $\dimn V_i / \dimn W_i = m+1$, one can first pick an independent group $\dmb{w}_1,\ldots,\dmb{w}_{m_i-1}$  in $V_i\setminus W_{m_i}$ by the induction hypothesis, where $\dmb{w}_j\in W_{ij}$ $(1\le j\le m_i-1)$.
Hence, $$\psi_{1m_i-1}(\dmb{w}_1),\ldots,\psi_{m_{i}-2m_i-1}(\dmb{w}_{m_{i}-2}), \dmb{w}_{m_i-1}\mbox{ are linear independent,}$$
where $\psi_{xy}\in\Hom_\dmf{G}(W_{ix},W_{iy})$. Let $\psi_{m_i-1m_i}$ be a bijection in $\Hom_\dmf{G}(W_{im_i-1},W_{im_i})$. Then one can obtain a linear independent group of vectors
$$\psi_{m_i-1m_i}\psi_{1m_i-1}(\dmb{w}_1),\ldots,
\psi_{m_i-1m_i}\psi_{m_{i}-2m_i-1}(\dmb{w}_{m_{i}-2}), \psi_{m_i-1m_i}(\dmb{w}_{m_i-1}).$$ Evidently, $\psi_{m_i-1m_i}\psi_{jm_i-1}$ is a bijection in $\Hom_\dmf{G}(W_{ij},W_{im_i})$, $1\le j\le m_i-2$. By means of Corollary \ref{Cor-Dimension}, $\dimn \Hom_\dmf{G}(W_{ij},W_{im_i})=1$ and thus there exists a scalar $\lambda_j\neq0$ so that $\psi_{m_i-1m_i}\psi_{jm_i-1}=\lambda_j\psi_{jm_i}$, $1\le j\le m_i-2$. Consequently, $$\lambda_1\psi_{1m_i}(\dmb{w}_1),\ldots,
\lambda_{m_i-2}\psi_{m_i-2m_i}(\dmb{w}_{m_{i}-2}), \psi_{m_i-1m_i}(\dmb{w}_{m_i-1})\mbox{ are linear independent.}$$ Since $\dimn W_{i}\ge\dimn V_i/\dimn W_{i}$ and $W_i\cong W_{im_i}$, one can find out a vector $\dmb{w}_{m_i}\in \dmb{w}_{m_i}$ so that $$\psi_{1m_i}(\lambda_1\dmb{w}_1),\ldots,
\psi_{m_i-2m_i}(\lambda_{m_i-2}\dmb{w}_{m_{i}-2}), \psi_{m_i-1m_i}(\dmb{w}_{m_i-1}), \dmb{w}_{m_i}\mbox{ are linear independent.}$$
Evidently, $\lambda_1\dmb{w}_1,\ldots,\lambda_{m_i-2}\dmb{w}_{m_{i}-2}, \dmb{w}_{m_i-1}, \dmb{w}_{m_i}$ is an independent group in $V_i$.
\end{proof}

According to approaches to establish Theorem \ref{Thm-1stCaseForEigenvector} and Theorem \ref{Thm-2ndCaseForEigenvector}, the following is relevant.
\begin{cor}\label{Cor-Fundamental} Suppose that $W_1,\ldots,W_h$ are the all distinct irreducible representations of $\rho$ and $\C^n=\oplus_{i=1}^hV_i$ is the canonical decomposition of the representation $\rho$.
Then there exists a vector $\dmb{v}$ in $\C^n$ so that
$$\dimn \spa(\dmf{G} \dmb{v})=\sum_{i=1}^h\min\{\dimn V_i,(\dimn W_i)^2\}.$$
\end{cor}
\begin{remark} According to the remark of Theorem \ref{Thm-2ndCaseForEigenvector},
$\sum_{i=1}^h\min\{\dimn V_i,(\dimn W_i)^2\}$ is the largest number $\dimn \spa(\dmf{G} v)$ can attain for any $\dmb{v}\in\C^n$.
\end{remark}

Finally, we determine $\dimn \spa(\dmf{G} \dmb{v})$ for any given vector $\dmb{v}$ in $\C^n$.
Let $\C^n=\oplus_{i=1}^h V_i$ be the canonical decomposition of the representation $\rho$. For each $V_i$ $(i=1,\ldots,h)$, we now fix a decomposition $V_i=\oplus_{j=1}^{m_i} W_{ij}$ into irreducible representations such that each $W_{ij}$ is isomorphic to $W_i$. Apparently, for every vector $\dmb{v}$ in $\C^n$, one can find a unique vector $\dmb{v}_i$ in $V_i$ $(i=1,\ldots,h)$ so that $\dmb{v}=\sum_{i=1}^h \dmb{v}_i$, and furthermore, each $\dmb{v}_i$ can be expressed uniquely as $\sum_{j=1}^{m_i} \dmb{w}_{ij}$ where $\dmb{w}_{ij}\in W_{ij}$ $(j=1,\ldots,m_i)$.

We use $n_i$ to denote the cardinality of maximum independent group in the group $\dmb{w}_{i1},\ldots,\dmb{w}_{im_i}$ $(1\le i\le h)$. It is not difficult to see that $$n_i\le\min\{\dimn W_i,\dimn V_i/\dimn W_i\}.$$
By means of approaches in proving Theorem \ref{Thm-2ndCaseForEigenvector}, one can readily see that
$$\dimn \spa(\dmf{G} \dmb{v}_i)=n_i\cdot\dimn W_i~~(i=1,\ldots,h).$$
As a result, according to the proof of Theorem \ref{Thm-1stCaseForEigenvector}, we have
$$\dimn \spa(\dmf{G} \dmb{v})=\sum_{i=1}^h \dimn \spa(\dmf{G} \dmb{v}_i)=\sum_{i=1}^h n_i\cdot\dimn W_i.$$

\section{Petersen graph}


In the last section, we characterize theoretically those extremal eigenvectors of $\dbf{A}(G)$ for a graph $G$ through the permutation representation of $\drm{Aut}~G$. In this section, we take the Petersen graph $Pet$ (see Fig. \ref{Fig-1}) as an example and compute those extremal eigenvectors for $Pet$. As we shall see below, $Pet$ enjoys an amazing property that each eigenspace of the adjacency matrix $\dbf{A}$ is an irreducible representation of the permutation representation of $\drm{Aut}~ Pet$. It is well-known that for any finite group $\dmf{G}$, one can build a graph $G$ such that $\drm{Aut}~G = \dmf{G}$, so the fact enjoyed by $Pet$ raises a further question whether we can build a graph $G$  for any finite group $\dmf{G}$ such that $\drm{Aut}~G = \dmf{G}$ and each eigenspace of $\dbf{A}(G)$ is an irreducible representation of the permutation representation of $\dmf{G}$. It is a notorious problem to compute all irreducible representations of the permutation representation for a finite group (see \cite{LP} for details), so it would be interesting if one can efficiently construct such a graph enjoying the property above.\vspace{2mm}

As well-known, it is not easy to determine the automorphism group of $Pet$, so for completeness, we present here an elegant way to obtain $\drm{Aut}~Pet$.

First of all, we label the vertices of $Pet$ by 10 subsets of $[5]=\{1,2,3,4,5\}$, each of which contains exactly 2 distinct elements of $[5]$, such that $\{x,y\}$ and $\{u,v\}$ are adjacent if and only if $\{x,y\}$ and $\{u,v\}$ have no common elements,  as shown on the right in Fig. \ref{Fig-1}. As a result, each permutation $\sigma$ in $\dmf{S}_5$ induces a bijective map $\tau_{\sigma}$ from the vertex set of $Pet$ to itself, and furthermore, since $\sigma$ is a bijective map from $[5]$ to itself, $$\{\sigma(x),\sigma(y)\}\cap\{\sigma(u),\sigma(v)\}=\emptyset \Leftrightarrow \{x,y\}\cap\{u,v\}=\emptyset,$$
{\it i.e.,} the map $\tau_{\sigma}$ keeps the adjacency relation between vertices of $Pet$.
Therefore, each permutation $\sigma$ in $\dmf{S}_5$ induces an automorphism $\tau_{\sigma}$ of $Pet$ and thus $\dmf{S}_5$ is isomorphic to a subgroup of $\drm{Aut}~Pet$. Consequently, $|\drm{Aut}~Pet|\ge 120$.

We now show, by means of the Orbit-stabilizer Lemma, that  $|\drm{Aut}~Pet|$ is indeed equal to 120, and thus $\drm{Aut}~Pet$ is isomorphic to $\dmf{S}_5$. Let $\dmf{G}$ be the automorphism group of $Pet$. Then the orbit $1^{\dmf{G}}$ contains all vertices of $Pet$, for Petersen graph is vertex-transitive. According  to the Orbit-Stabilizer Lemma,
$$|\dmf{G}|=|1^{\dmf{G}}||\dmf{G}_1|=10\cdot|\dmf{G}_1|.$$
Consider now the orbit of vertex 6 acted by $\dmf{G}_1$. Since the vertex 6 is adjacent to the vertex 1 and $\dmf{G}_1$ fixes the vertex 1, one can readily see that $6^{\dmf{G}_1}=\{2,5,6\}$ and thus $$|\dmf{G}_1|=|6^{\dmf{G}_1}||\dmf{G}_{1,6}|=3\cdot|\dmf{G}_{1,6}|.$$
Again, we choose another vertex 2 adjacent to the vertex 1 and consider the orbit of 2 acted by $\dmf{G}_{1,6}$. According to a similar reason, one can readily see that $|2^{\dmf{G}_{1,6}}|=\{2,5\}$ and thus $$|\dmf{G}_{1,6}|=|2^{\dmf{G}_{1,6}}||\dmf{G}_{1,2,6}|=2\cdot|\dmf{G}_{1,2,6}|.$$
It is easy to see there is only one non-trivial element in $\dmf{G}_{1,2,6}$, namely, $(3,7)(4,10)(8,9)$. Consequently, $|\dmf{G}_{1,2,6}|=2$. Therefore, $|\dmf{G}|=120$ and thus $\dmf{G}\cong\dmf{S}_5$.

It is easy to check by means of the right part of Fig.\ref{Fig-1} that two permutations $(1,2,3,4,5)$ and $(1,2)$ in $\dmf{S}_5$ induce two automorphisms of $Pet$, namely, $(1,4,2,5,3)(6,9,7,10,8)$ and $(3,7)(4,10)(8,9)$. Since $\dmf{S}_5$ is generated by the two elements $(1,2,3,4,5)$ and $(1,2)$ and $\dmf{G}\cong\dmf{S}_5$, $\dmf{G}$ can be generated by the two automorphisms $\sigma_1=(1,4,2,5,3)(6,9,7,10,8)$ and $\sigma_2=(3,7)(4,10)(8,9)$. As a result, to determine the irreducible representations of the permutation representation $\rho$ of the automorphism group $\dmf{G}$, which maps each permutation $\sigma$ in $\dmf{G}$ to its corresponding linear operator $\dcal{T}_{\sigma}$ on $\C^{10}$, it suffices to figure out those minimum common invariant subspaces of $\dcal{T}_{\sigma_1}$ and $\dcal{T}_{\sigma_2}$. For convenience, we use $\sigma_i$ $(i=1,2)$ to represent the linear operator $\dcal{T}_{\sigma_i}$ in what follows. \vspace{2mm}

It is well-known that the Petersen graph is a strong regular graph and its adjacency matrix $\dbf{A}$ has 3 distinct eigenvalues 3, 1 and -2 of multiplicity 1, 5 and 4, respectively. Let $V_3$, $V_1$ and $V_{-2}$ denote the three eigenspaces of $\dbf{A}$ corresponding to the three eigenvalues respectively. We now figure out those those extremal eigenvectors in every eigenspace.

The eigenspace $V_3$ is spanned by the vector $\bf\dmb{1}$, every entry of which is equal to the number 1, and thus $V_3$ is an irreducible representation of $\rho$. Consequently, every vector in $V_3$ is both symmetric and asymmetric.

Since $V_1$ and $V_{-2}$ are both $\sigma_1$-invariant, each of them has a basis consisting of eigenvectors of $\sigma_1$ in virtue of Lemma \ref{Lem-InvariantSubspace}. To obtain those eigenvectors of $\sigma_1$, we first build the matrix $\dbf{X}_1$ of $\sigma_1$ with respect to the basis $\dmb{b}_1=\dmb{1},\dmb{b}_2,\ldots,\dmb{b}_{10}$ consisting of eigenvectors of $\dbf{A}$, where the vectors $\dmb{b}_2,\ldots,\dmb{b}_6$ span the eigenspace $V_1$ and the vectors  $\dmb{b}_7,\ldots,\dmb{b}_{10}$ span the eigenspace $V_{-2}$.

As we have seen in the first section, $\dbf{P}_1$ is the matrix of $\sigma_1$ with respect to the standard basis $\dmb{e}_1,\ldots,\dmb{e}_{10}$, so $$\dbf{X}_1=\dbf{B}^{-1}\dbf{P}_1\dbf{B},$$
where $\dbf{B}=(\dmb{1},\dmb{b}_2,\ldots,\dmb{b}_{10})$ is the transition matrix from the standard basis to the basis $\dmb{1},\dmb{b}_2,\ldots,\dmb{b}_{10}$. Then the matrix $\dbf{X}_1$ consists of three blocks so that each of them corresponds to a eigenspace of $\dbf{A}$. We now compute the eigenvectors $\dmb{x}_{1},\ldots,\dmb{x}_{10}$ of $\dbf{X}_1$. Obviously, the expression $\dmb{x}_{i}$ $(i=1,\ldots,10)$ is relative to the basis $\dmb{1}_1,\ldots,\dmb{b}_{10}$, and the expression $\dmb{u}_{i}=\dbf{B}\dmb{x}_{i}$ $(i=1,\ldots,10)$ is relative to the standard basis. Consequently, $V_1=\drm{span}(\dmb{u}_{2},\ldots,\dmb{u}_{6})$ and $V_{-2}=\drm{span}(\dmb{u}_{7},\ldots,\dmb{u}_{10})$.

On the other hand, it is easy to see that $\sigma_1$ has five distinct eigenvalues $1, e^{\frac{2\pi}{5}\ii}, e^{\frac{4\pi}{5}\ii},$ $e^{\frac{6\pi}{5}\ii}, e^{\frac{8\pi}{5}\ii}$, and each of them has multiplicity 2. The further computation shows that each eigenspace of $\sigma_1$ in $V_1$ (or $V_{-2}$) is of dimension one. We do all those numerical computations by means of \textbf{\textit{Mathematica}}.\vspace{2mm}

Five eigenvectors of $\sigma_1$ in $V_1$, corresponding to eigenvalues $1, e^{\frac{2\pi}{5}\ii}, e^{\frac{4\pi}{5}\ii},$ $e^{\frac{6\pi}{5}\ii}, e^{\frac{8\pi}{5}\ii}$, respectively, are listed below.
$$
(\dmb{u}_{2},\dmb{u}_{3},\dmb{u}_{4}) =
\begin{pmatrix}
   -1 & 2.61803                 & 0.381966 + 1.11022\times 10^{-16}\ii  \\
   -1 & 0.809017 - 2.4899\ii    & -0.309017 - 0.224514\ii  \\
   -1 & -2.11803 - 1.53884\ii   & 0.118034 + 0.363271\ii  \\
   -1 & -2.11803 + 1.53884\ii   & 0.118034 - 0.363271\ii \\
   -1 & 0.809017 + 2.4899\ii    & -0.309017 + 0.224514\ii \\
    1 & 1                       & 1 \\
    1 & 0.309017 - 0.951057 \ii & -0.809017 - 0.587785\ii   \\
    1 & -0.809017 - 0.587785\ii & 0.309017 + 0.951057\ii \\
    1 & -0.809017 + 0.587785\ii & 0.309017 - 0.951057\ii  \\
    1 & 0.309017 + 0.951057\ii  & -0.809017 + 0.587785\ii  \\
  \end{pmatrix},
$$
and $\dmb{u}_{5}=\overline{\dmb{u}_{4}}$ and $\dmb{u}_{6}=\overline{\dmb{u}_{3}}$.\vspace{2mm}

Four eigenvectors of $\sigma_1$ in $V_{-2}$, corresponding to eigenvalues $ e^{\frac{2\pi}{5}\ii}, e^{\frac{4\pi}{5}\ii},$ $e^{\frac{6\pi}{5}\ii}, e^{\frac{8\pi}{5}\ii}$, respectively, are listed below.
$$(\dmb{u}_{7},\dmb{u}_{8}) =
\begin{pmatrix}
    -0.118034 - 0.363271\ii & 2.11803 - 1.53884\ii        \\
    -0.381966 & -2.61803 + 1.11022\times 10^{-16}\ii      \\
    -0.118034 + 0.363271\ii & 2.11803 + 1.53884\ii        \\
    0.309017 + 0.224514\ii  & -0.809017 - 2.4899\ii       \\
    0.309017 - 0.224514\ii  & -0.809017 + 2.4899\ii       \\
    0.309017 + 0.951057\ii  & -0.809017 + 0.587785\ii     \\
    1                           & 1                       \\
    0.309017 - 0.951057\ii  & -0.809017 - 0.587785\ii     \\
    -0.809017 - 0.587785\ii & 0.309017 + 0.951057\ii      \\
    -0.809017 + 0.587785\ii & 0.309017 - 0.951057\ii      \\
  \end{pmatrix},
$$
and $\dmb{u}_{9}=\overline{\dmb{u}_{8}}$ and $\dmb{u}_{10}=\overline{\dmb{u}_{7}}$.\vspace{2mm}

Now we show that the eigenspace $V_1$ is an irreducible representation of $\rho$. As we pointed out before, it is sufficient to show that $V_1$ is the minimum common invariant subspaces of $\sigma_1$ and $\sigma_2$.  Suppose $U$ is a $\dmf{G}$-invariant subspace in $V_1$ spanned by $\dmb{u}_{i_1},\ldots,\dmb{u}_{i_s}$, where $i_j\in\{2,\ldots,6\}$. Then  according to the definition of $\dmf{G}$-invariant subspace, two vectors $\sigma_2\dmb{u}_{i_j}$ and $\sigma_1^{l}\circ\sigma_2\dmb{u}_{i_j}$ $(j=1,\ldots,s)$ must be in $U$ where $l$ is a positive integer not more than 5. On the other hand, it is not difficult, by means of \textbf{\textit{Mathematica}}, to check that five vectors $\sigma_2\dmb{u}_{i}$, $\sigma_1^{1}\circ\sigma_2\dmb{u}_{i},\ldots, \sigma_1^{4}\circ\sigma_2\dmb{u}_{i}$ $(i=2,\ldots,6)$ are linear independent. Thus the dimension of $U$ is 5 due to Lemma \ref{Lem-NonInvariantSubspace}, and therefore, $V_1$ is an irreducible representation of $\rho$.

One can prove the eigenspace $V_{-2}$ is also an irreducible representation of $\rho$ in a similar way. As a result, every vector in $V_1$ (or in $V_{-2}$) is both symmetric and asymmetric.

The symmetric and asymmetric vectors of an eigenspace of $\dbf{A}(G)$ so far are defined for the automorphism group $\drm{Aut}~G$, but evidently those notions can be defined for any subgroup of $\drm{Aut}~G$. The right part of Fig.\ref{Fig-2} illustrates a symmetric eigenvector of $V_1$ for the cyclic group $\langle\sigma_1\rangle$, and the right part of Fig.\ref{Fig-3} illustrates a symmetric eigenvector of $V_{-2}$ for the cyclic group $\langle\sigma_2\rangle$.

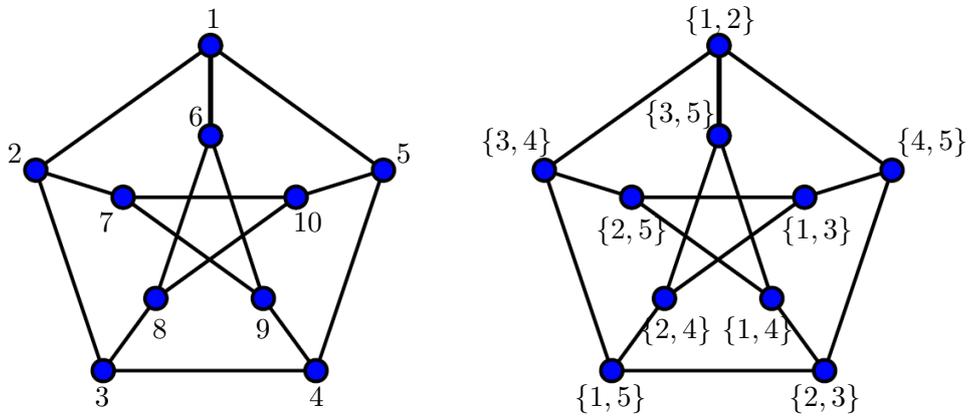
\begin{figure}[!htbp]
\begin{center}\setlength{\unitlength}{0.5bp}
 \includegraphics[width=13.3cm]{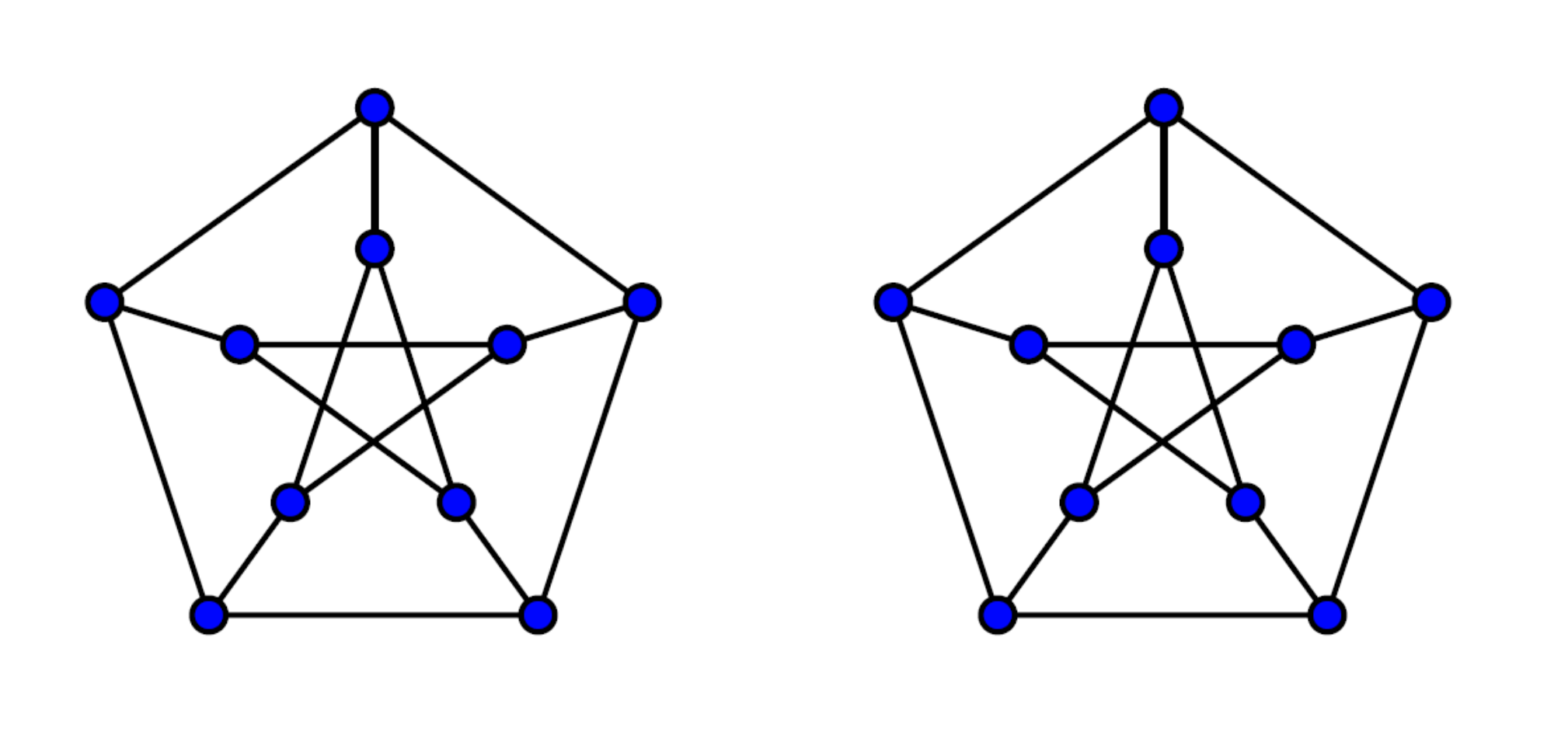}
 \put(-577,318){\fontsize{11}{2.73}\selectfont $1$}
 \put(-725,217){\fontsize{11}{2.73}\selectfont $2$}
 \put(-660,34){\fontsize{11}{2.73}\selectfont $3$}
 \put(-500,34){\fontsize{11}{2.73}\selectfont $4$}
 \put(-435,217){\fontsize{11}{2.73}\selectfont $5$}
 \put(-590,245){\fontsize{11}{2.73}\selectfont $6$}
 \put(-657,165){\fontsize{11}{2.73}\selectfont $7$}
 \put(-617,85){\fontsize{11}{2.73}\selectfont $8$}
 \put(-540,85){\fontsize{11}{2.73}\selectfont $9$}
 \put(-512,165){\fontsize{11}{2.73}\selectfont $10$}
\put(-220,318){\fontsize{11}{2.73}\selectfont $\{1,2\}$}
 \put(-148,160){\fontsize{11}{2.73}\selectfont $\{1,3\}$}
 \put(-192,85){\fontsize{11}{2.73}\selectfont $\{1,4\}$}
 \put(-302,34){\fontsize{11}{2.73}\selectfont $\{1,5\}$}
  \put(-142,34){\fontsize{11}{2.73}\selectfont $\{2,3\}$}
   \put(-253,85){\fontsize{11}{2.73}\selectfont $\{2,4\}$}
   \put(-286,160){\fontsize{11}{2.73}\selectfont $\{2,5\}$}
   \put(-372,227){\fontsize{11}{2.73}\selectfont $\{3,4\}$}
   \put(-250,248){\fontsize{11}{2.73}\selectfont $\{3,5\}$}
   \put(-62,227){\fontsize{11}{2.73}\selectfont $\{4,5\}$}
\end{center}
\caption{ Petersen graph $Pet$}\label{Fig-1}
\end{figure}

\begin{figure}[!htbp]
\begin{center}\setlength{\unitlength}{0.5bp}
 \includegraphics[width=13.3cm]{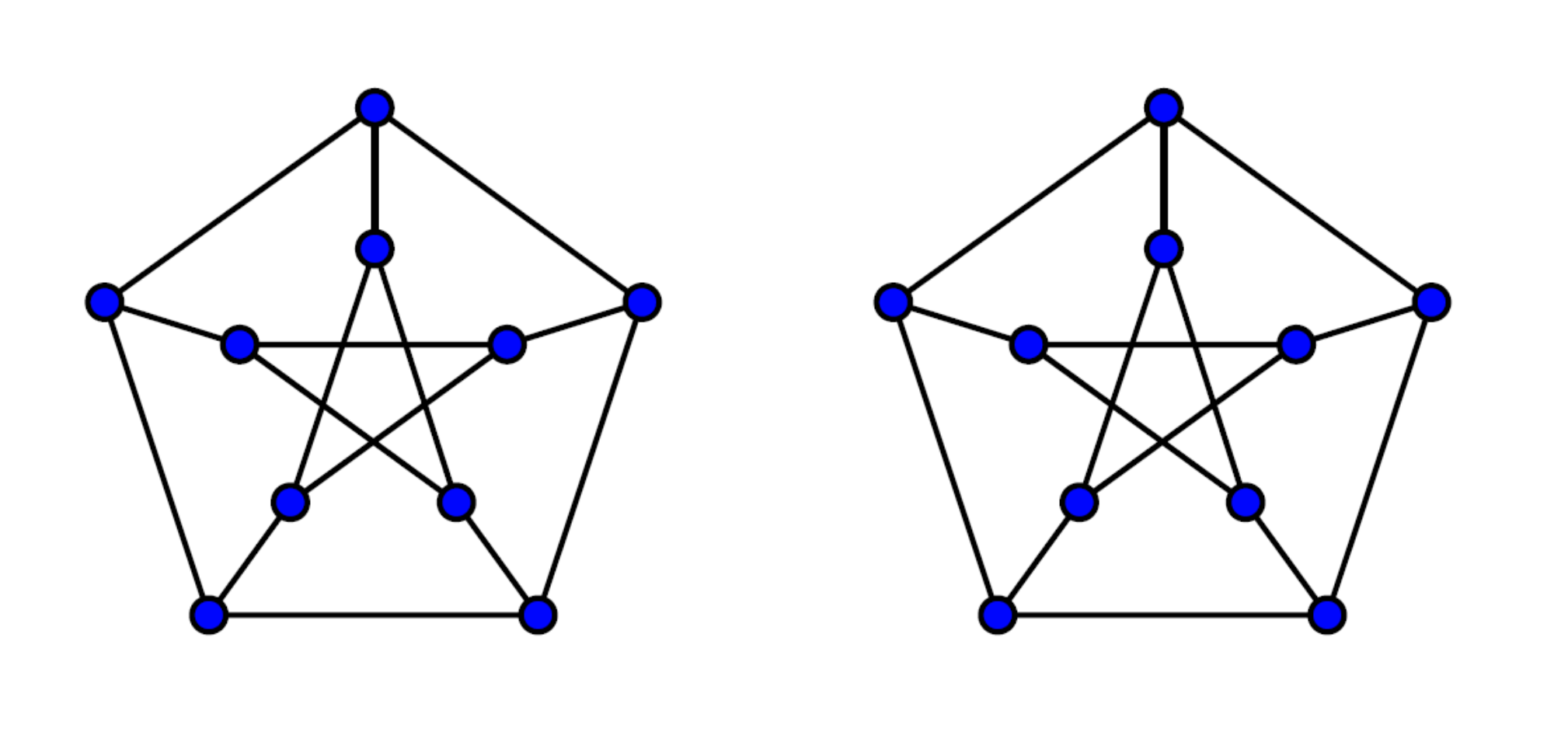}
 \put(-577,318){\fontsize{11}{2.73}\selectfont $1$}
 \put(-727,225){\fontsize{11}{2.73}\selectfont $-1$}
 \put(-660,34){\fontsize{11}{2.73}\selectfont $-2$}
 \put(-500,34){\fontsize{11}{2.73}\selectfont $-1$}
 \put(-435,217){\fontsize{11}{2.73}\selectfont $3$}
 \put(-610,245){\fontsize{11}{2.73}\selectfont $-1$}
 \put(-645,160){\fontsize{11}{2.73}\selectfont $0$}
 \put(-617,85){\fontsize{11}{2.73}\selectfont $0$}
 \put(-560,85){\fontsize{11}{2.73}\selectfont $-2$}
 \put(-516,160){\fontsize{11}{2.73}\selectfont $3$}
\put(-202,318){\fontsize{11}{2.73}\selectfont $-1$}
\put(-335,227){\fontsize{11}{2.73}\selectfont $-1$}
\put(-282,33){\fontsize{11}{2.73}\selectfont $-1$}
\put(-123,34){\fontsize{11}{2.73}\selectfont $-1$}
\put(-56,222){\fontsize{11}{2.73}\selectfont $-1$}
\put(-210,246){\fontsize{11}{2.73}\selectfont $1$}
\put(-266,160){\fontsize{11}{2.73}\selectfont $1$}
\put(-240,85){\fontsize{11}{2.73}\selectfont $1$}
\put(-162,85){\fontsize{11}{2.73}\selectfont $1$}
 \put(-136,160){\fontsize{11}{2.73}\selectfont $1$}
 \put(-350,-10){\fontsize{11}{2.73}\selectfont A symmetric eigenvector of $V_1$ for $\langle\sigma_1\rangle$}
\end{center}
\caption{Asymmetric eigenvectors of $V_1$ for $\drm{Aut}~Pet$}\label{Fig-2}
\end{figure}

\begin{figure}[!htbp]
\begin{center}\setlength{\unitlength}{0.5bp}
 \includegraphics[width=13.3cm]{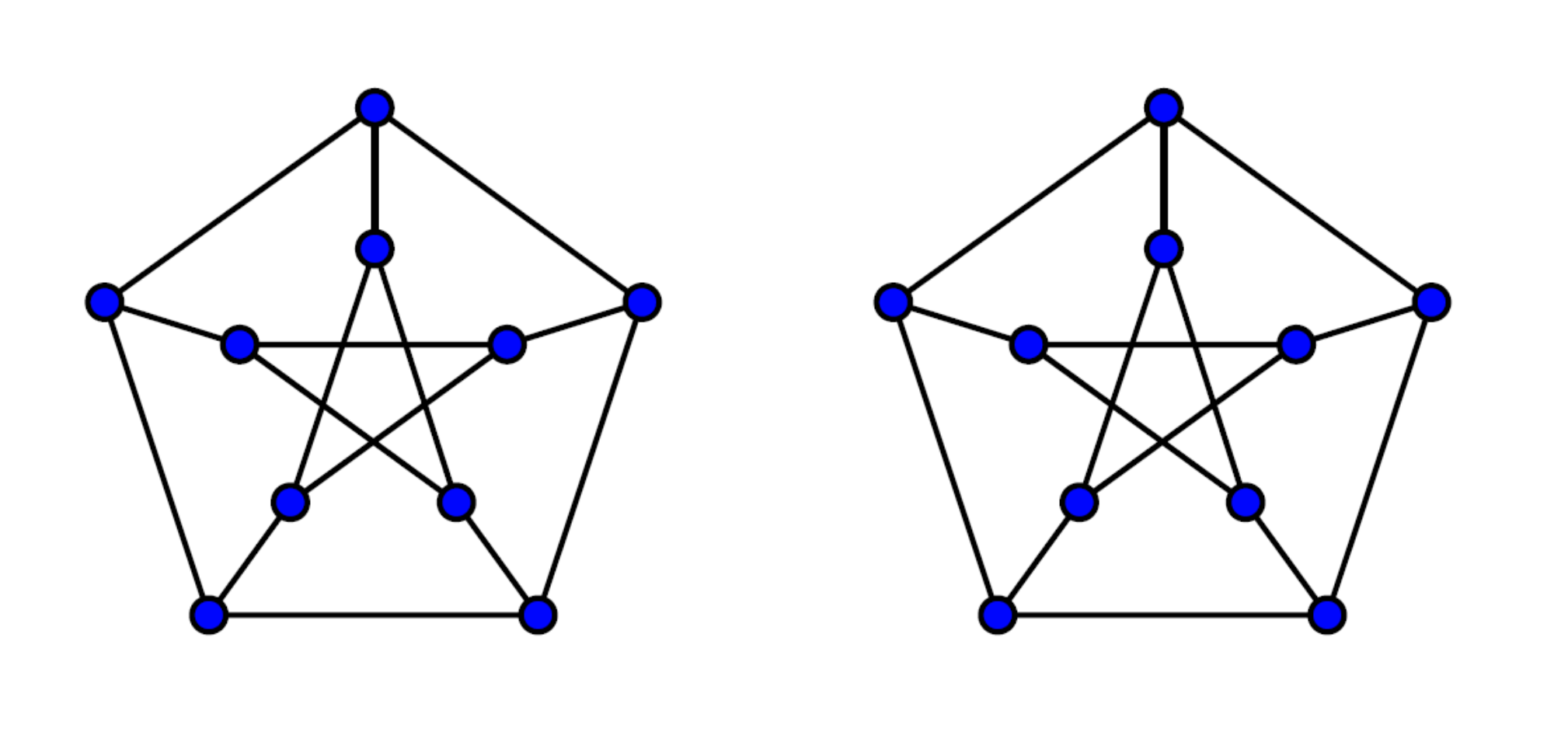}
 \put(-577,318){\fontsize{11}{2.73}\selectfont $4$}
 \put(-725,225){\fontsize{11}{2.73}\selectfont $-6$}
 \put(-660,34){\fontsize{11}{2.73}\selectfont $4$}
 \put(-500,34){\fontsize{11}{2.73}\selectfont $-1$}
 \put(-435,217){\fontsize{11}{2.73}\selectfont $-1$}
 \put(-610,245){\fontsize{11}{2.73}\selectfont $-1$}
 \put(-657,165){\fontsize{11}{2.73}\selectfont $4$}
 \put(-627,85){\fontsize{11}{2.73}\selectfont $-1$}
 \put(-555,85){\fontsize{11}{2.73}\selectfont $-1$}
 \put(-512,165){\fontsize{11}{2.73}\selectfont $-1$}
\put(-202,318){\fontsize{11}{2.73}\selectfont $0$}
\put(-335,227){\fontsize{11}{2.73}\selectfont $0$}
\put(-282,33){\fontsize{11}{2.73}\selectfont $-1$}
\put(-123,34){\fontsize{11}{2.73}\selectfont $1$}
\put(-56,222){\fontsize{11}{2.73}\selectfont $0$}
\put(-210,246){\fontsize{11}{2.73}\selectfont $0$}
\put(-272,160){\fontsize{11}{2.73}\selectfont $1$}
\put(-240,85){\fontsize{11}{2.73}\selectfont $1$}
\put(-175,85){\fontsize{11}{2.73}\selectfont $-1$}
 \put(-136,160){\fontsize{11}{2.73}\selectfont $-1$}
 \put(-350,-10){\fontsize{11}{2.73}\selectfont A symmetric eigenvector of $V_1$ for $\langle\sigma_2\rangle$}
\end{center}
\caption{Asymmetric eigenvectors of $V_{-2}$ for $\drm{Aut}~Pet$}\label{Fig-3}
\end{figure}

\end{document}